\documentclass[10pt]{article}  

\usepackage{amsmath, amsthm, amsfonts, amssymb}
\usepackage[bookmarks]{hyperref}
\usepackage{indentfirst} 
\usepackage{marvosym}
\usepackage{enumitem}
\usepackage[a4paper]{geometry}

\newtheorem{lemma}{Lemma}[section]

\newtheorem{proposition}[lemma]{Proposition}

\renewenvironment{proof}[1][\proofname]{{\noindent\bf #1. }}{\qed}
\newtheorem*{maintheorem}{Main Theorem}

\theoremstyle{definition}
\newtheorem{remark}[lemma]{Remark}

{\bf}{\rm}

\newcommand{\abs}[1]{\ensuremath{|#1|}}
\newcommand{\op}{\operatorname}
\newcommand{\ce}[2]{\pmb{\op{C}}_{#1}(#2)}

\newcommand{\ze}[1]{\pmb{\op{Z}}(#1)}

\begin{document}

\title{\bf A completeness criterion for the common divisor graph on $p$-regular class sizes}

\author{\sc V. Sotomayor\thanks{Departamento de Álgebra, Facultad de Ciencias, Universidad de Granada, Av. Fuente Nueva s/n, 18071 Granada, Spain. \newline
\Letter: \texttt{vsotomayor@ugr.es} \newline 
ORCID: 0000-0001-8649-5742 \newline \rule{6cm}{0.1mm}\newline
This research is supported by Ayuda a Primeros Proyectos de Investigación (PAID-06-23), Vicerrectorado de Investigación de la Universitat Politècnica de València (UPV). \newline
}}

\date{}

\maketitle

\begin{abstract}
\noindent Let $G$ be a finite group. For a fixed prime $p$, let $\Gamma_p(G)$ be the common divisor graph built on the set of sizes of $p$-regular conjugacy classes of $G$: this is the simple undirected graph whose vertices are the class sizes of those non-central elements of $G$ such that $p$ does not divide their order, and two distinct vertices are adjacent if and only if they are not coprime. In this note we prove that if $\Gamma_p(G)$ is a $k$-regular graph with $k\geq 1$, then it is a complete graph with $k+1$ vertices.

\medskip

\noindent \textbf{Keywords} Finite groups $\cdot$ Conjugacy classes $\cdot$ $p$-regular elements $\cdot$ Regular graphs $\cdot$ Complete graphs

\smallskip

\noindent \textbf{2020 MSC} 20E45 
\end{abstract}


\section{Introduction}

Every group considered in the following discussion is tacitly assumed to be a finite group. A long-standing area of research within group theory is the study of the interconnection that exists between the algebraic structure of a group and the arithmetical properties of certain positive integers associated to it. In this framework, several graphs have been introduced to get a better understanding of these arithmetical features. For instance, if $X$ is a set of positive integers, then the \emph{common divisor graph} $\Gamma(X)$ is the simple undirected graph whose vertices are the numbers in $X$ strictly greater than 1, and two vertices are adjacent whenever they have a common prime divisor. Another graph that has revealed to be very useful is the \emph{prime graph} $\Delta(X)$, whose vertices are the prime divisors of any number in $X$, and two vertices are adjacent whenever there is some number in $X$ which is divisible by both primes. We refer the interested reader to the excellent survey \cite{Lewis} for results on this topic when $X$ is either the set of irreducible complex character degrees, or the set of class sizes.

Several variations on the theme have appeared from the nineties up to nowadays. For instance, instead of focusing on the full set $X=cs(G)$ of class sizes of a group $G$, some smaller subsets have been considered. The goal is similar: to examine whether some local information about the group structure can still be read off from graph-theoretical properties of the corresponding subgraphs of $\Gamma(G):=\Gamma(cs(G))$ or $\Delta(G):=\Delta(cs(G))$. For example, the subgraphs induced by the subset of class sizes of elements lying in a normal subgroup of $G$ give information on the normal structure of $G$ (\emph{cf.} \cite{BFM, RA}). Another instance of this fact is provided by the subset $cs_p(G)$ of class sizes of $p$-regular elements of $G$, for a given prime $p$, since the respective induced subgraphs $\Gamma_p(G):=\Gamma(cs_p(G))$ and $\Delta_p(G):=\Delta(cs_p(G))$ encode information about the $p$-structure of $G$ when it is $p$-soluble (see \cite{BF1, BF2, BFsurvey, LZ} and Remark \ref{graphs}).

In this field of study, a central question that arises is the following: which sets of positive integers can occur as $cs_p(G)$ for some group $G$ and some prime $p$? The aim of this paper is to give an answer to this question via the following criterion for the graph $\Gamma_p(G)$ to be complete.

\begin{maintheorem}
\label{teoA}
Let $G$ be a group, and $p$ be a prime. Then $\Gamma_p(G)$ is a $k$-regular graph for some integer $k\geq 1$ if and only if it is a complete graph with $k+1$ vertices.
\end{maintheorem}

We highlight that, in contrast to all the results within the literature of graphs on $p$-regular conjugacy classes (\emph{cf.} \cite{BF1, BF2, LZ}), our Main Theorem is also valid for groups that are not $p$-soluble. Usually, the $p$-solubility assumption is required only to ensure the next feature: \emph{the product $BC$ of two $p$-regular classes $B$ and $C$ of $G$ such that $(|B|, |C|)=1$ is again a $p$-regular conjugacy class of $G$}. The reader can find a proof of this fact when $G$ is $p$-soluble in \cite[Lemma 1]{LZ}. Observe that the product $BC$ is certainly a conjugacy class of $G$, since $B$ and $C$ have coprime sizes. Thus the role of $p$-solubility in \cite[Lemma 1]{LZ} is only to guarantee that some (and so, any) element in $BC$ is $p$-regular whenever $B$ and $C$ are $p$-regular coprime classes. Very recently, this has been demonstrated (see \cite[Theorem D]{many}) via the classification of finite simple groups. Therefore, the $p$-solubility assumption can be avoided in many results within the literature of this topic: see Section 4 of \cite{many} for a collection of them, including our Main Theorem.

As a consequence of the Main Theorem when we take a prime $p$ which does not divide the group order, we recover the main result of \cite{BCHP} concerning the ordinary graph $\Gamma(G)$, which was firstly proved in \cite{BHPS} for values $k\in\{2,3\}$. We remark that the authors of \cite{BCHP} made use of the characterisation provided in \cite{K} about groups $G$ such that the diameter of $\Gamma(G)$ is exactly three. However, our approach utilises different techniques, since the corresponding situation for the diameter of the subgraph $\Gamma_p(G)$ on $p$-regular class sizes has not yet been settled. In fact, a conjecture regarding the $p$-structure of such groups, and a partial positive answer, has recently been established in \cite{FJS}.

Observe that the subgraph $\Gamma_p(G)$ may be regular even if the ordinary graph $\Gamma(G)$ is not: for instance, if $p=2$ and $G=A\Gamma(\mathbb{F}_8)$ is an affine semilinear group over the field of eight elements, then $cs(G)=\{1, 7, 24, 28\}$ so $\Gamma(G)$ is certainly non-regular (and non-complete), but $cs_2(G)=\{1,24,28\}$ and thus $\Gamma_2(G)$ is regular (and so complete).

It is well-known that the graphs $\Gamma(X)$ and $\Delta(X)$ defined over a set $X$ of positive integers share some significant features: they possess the same number of connected components, and their diameters differ by at most one (see \cite[Corollary 3.2]{Lewis}). Nonetheless, the subgraph $\Delta_p(G)$ has a different behaviour with respect to the situation studied in the Main Theorem, even when $cs_p(G)=cs(G)$, since there exist groups $G$ such that $\Delta(G)$ is connected, regular and non-complete, as it was pointed out in \cite{BCHP}.


\section{The results}

Hereafter, we write $\pi(G)$ for the set of prime divisors of $\abs{G}$, and if $a^G$ is a conjugacy class of $G$, then $\pi(a^G)$ is the set of prime divisors of $|a^G|=|G:\ce{G}{a}|$. Recall that each element $g\in G$ can be decomposed as product of pairwise commuting elements of prime power order, say $g_{q_1},...,g_{q_s}$, for some integer $s\geq 1$ and certain primes $\{q_1,...,q_s\}$, and each $q_i$-element $g_{q_i}$ is called the $q_i$-part of $g$; we call this the \emph{primary decomposition} of $g$. The greatest common divisor of two positive integers $a$ and $b$ is denoted by $(a,b)$. If $\pi$ is a set of primes, then the largest $\pi$-number that divides a positive integer $n$ is $n_{\pi}$.

A clique $C$ of a graph $\Gamma$ is a subset of vertices such that every two distinct vertices of $C$ are adjacent in $\Gamma$, \emph{i.e.} the subgraph of $\Gamma$ induced by $C$ is a complete graph. The vertices of a graph $\Gamma$ that are adjacent to a given vertex $v$ are called the neighbours of $v$, and the set of neighbours of $v$ is denoted by $\mathcal{N}(v)$. Two vertices $v,w$ of $\Gamma$ are called closed twins (or simply twins) if $\mathcal{N}(v)\cup\{v\}=\mathcal{N}(w)\cup\{w\}$. This concept clearly defines an equivalence relation on the vertex set of $\Gamma$.

The remaining notation and terminology used are standard in the frameworks of group theory and graph theory.

Let us start by proving three preliminary lemmas. In the first one, we will make use of certain theorems in \cite{BF2}, and so the next observation is important.

\begin{remark}
\label{graphs}
We point out that the common divisor graph $\Gamma_p(G)$ slightly differs with the graph on $p$-regular conjugacy classes mentioned in \cite{BF1, BF2, BFsurvey}: in those papers, the vertices are the non-central classes of $p$-regular elements of $G$, and there is an edge between two classes if their cardinalities have a common prime divisor. Nevertheless, it is easy to realise that this last graph can be transformed into $\Gamma_p(G)$ by collapsing all the classes with the same size into a single vertex; thus they have the same number of connected components, and the diameter of each component will be the same (except if it contains classes all of the same size). 
\end{remark}

\begin{lemma}
\label{connected}
Let $G$ be a group, and $p$ be a prime. If $\Gamma_p(G)$ is a $k$-regular graph for some $k\geq 1$, then $\Gamma_p(G)$ is connected.
\end{lemma}

\begin{proof}
Arguing by contradiction, let us suppose that $\Gamma_p(G)$ is non-connected. In virtue of \cite[Corollary 4.1]{many} we know that $\Gamma_p(G)$ has two connected components, say $X_1$ and $X_2$. Set $\pi_i$ for the union of $\pi(C)$ for every $C\in X_i$, where $i\in \{1,2\}$. We may assume that the maximum size of a $p$-regular class lies in $X_2$. We are going to distinguish certain cases depending on where the prime $p$ lies.

If $p\notin \pi_1\cup\pi_2$, then by \cite[Corollary 4.3 i)]{many} we obtain that $G$ has a $p$-complement $H$ as a direct factor, $H/\ze{H}$ is a Frobenius group, and the inverse images in $H$ of the kernel and a complement of $H/\ze{H}$ are abelian. In particular each component of $\Gamma_p(G)=\Gamma(H)$ is a single vertex, which contradicts the assumption of $k$-regularity with $k\geq 1$. If $p\in\pi_1$, then by \cite[Theorem 8]{BF2} ---the $p$-solubility is actually not needed there, see \cite[Corollary 4.3 ii)]{many}--- it follows that one component of $\Gamma_p(G)$ consists of a single vertex, a contradiction again. Hence we may suppose that $p\in\pi_2$. Let $a^G\in X_1$ and $b^G\in X_2$, and let $\ze{G}_{p'}$ be the $p$-complement of $\ze{G}$. Then $$|G:\ze{G}_{p'}|=|G:\ce{G}{b}|\cdot |\ce{G}{b}:\ce{G}{b}\cap\ce{G}{a}|\cdot |\ce{G}{b}\cap\ce{G}{a}:\ze{G}_{p'}|.$$ Applying \cite[Theorem 4 (a) and (c)]{BF2} ---the $p$-solubility is not needed there either, due to \cite[Theorem D]{many}--- we deduce that $|\ce{G}{a}:\ze{G}_{p'}|$ is a $\pi_2$-number, and therefore $|\ce{G}{b}\cap\ce{G}{a}:\ze{G}_{p'}|$ so is. Since $G=\ce{G}{a}\ce{G}{b}$, it follows $|a^G|=|\ce{G}{b}:\ce{G}{b}\cap\ce{G}{a}|=|G:\ze{G}_{p'}|_{\pi_1}$. Consequently $X_1$ consists of a single vertex, which is not possible.
\end{proof}

\smallskip

The following two lemmas adapt the ideas of \cite[Lemmas 2.3 and 2.4]{BCHP} for the graph $\Gamma_p(G)$. We include their proofs for the sake of thoroughness.

\begin{lemma}
\label{prime_power}
Let $G$ be a group, and $p$ be a prime. Suppose that $\Gamma_p(G)$ is $k$-regular for some $k\geq 1$. If there exists a non-central $p$-regular element $x\in G$ such that $|x^G|$ is a prime power, then $\Gamma_p(G)$ is a complete graph.
\end{lemma}

\begin{proof}
If $|x^G|$ is a power of a prime $q$ (which may be equal to or different from $p$), then its $k$ neighbours are also divisible by $q$, and thus $\mathcal{N}(|x^G|)\cup\{|x^G|\}$ is a clique of $\Gamma_p(G)$. Let us suppose, by contradiction, that $\Gamma_p(G)$ is not complete. Since $\Gamma_p(G)$ is connected by the previous lemma, there exists a $p$-regular element $y\in G\smallsetminus \ze{G}$ such that $|y^G|$ not divisible by $q$ but $|y^G|$ is adjacent to a neighbour of $|x^G|$. That neighbour has therefore degree at least $k+1$, which is certainly not possible.
\end{proof}

\begin{lemma}
\label{power}
Let $G$ be a group, and $p$ be a prime. Suppose that $\Gamma_p(G)$ is regular. Let $x\in G\smallsetminus\ze{G}$ be a $p$-regular element. If $y=x^{\alpha}\notin\ze{G}$ for some integer $\alpha$, then $|x^G|$ and $|y^G|$ are either equal or twins in $\Gamma_p(G)$.
\end{lemma}

\begin{proof}
Suppose that $|x^G|\neq |y^G|$. Since $\ce{G}{x}\leqslant\ce{G}{y}$, then $|y^G|$ divides $|x^G|$, so $\mathcal{N}(|y^G|)\subseteq \mathcal{N}(|x^G|)$. But the graph is regular by assumptions, which implies $\mathcal{N}(|y^G|)= \mathcal{N}(|x^G|)$. Moreover, as they are adjacent vertices of $\Gamma_p(G)$, it follows $\mathcal{N}(|y^G|)\cup\{|y^G|\}= \mathcal{N}(|x^G|)\cup\{|x^G|\}$. Hence $|y^G|$ and $|x^G|$ are twins in $\Gamma_p(G)$.
\end{proof}

\medskip

As mentioned in the Introduction, the proof of our Main Theorem differs from that originally appeared in \cite{BCHP}. We have instead adapted some arguments of \cite{RA}, as the proposition below.

\begin{proposition}
\label{key}
Let $G$ be a group, and $p$ be a prime. Suppose that $|G/\ze{G}|_{p'}$ is divisible by two primes $q\neq r$. Let $x_0,y_0\in G\smallsetminus\ze{G}$ be a $q$-element and a $r$-element, respectively, such that $x_0y_0=y_0x_0$. Besides, suppose that $\Gamma_p(G)$ is non-complete and regular. Then the following statements hold:
\begin{enumerate}
\setlength{\itemsep}{-1mm}
\item[\emph{(a)}] There exists a $q$-element $x_1\in G\smallsetminus\ze{G}$ and a $r$-element $y_1\in G\smallsetminus\ze{G}$ such that $|x_1^G|,|y_1^G|\in \mathcal{N}(|(x_0y_0)^G|)$ with $(|x_1^G|,|y_1^G|)=1$, $(|x_1^G|,qr)=r$ and $(|y_1^G|,qr)=q$. 
\item[\emph{(b)}] $|x_0^G|$ is divisible by $r$, $|y_0^G|$ is divisible by $q$, and $|(x_0y_0)^G|$ is divisible by $qr$.
\end{enumerate}
\end{proposition}

\begin{proof}
(a) Consider the subset $\{|x_0^G|, |y_0^G|, |(x_0y_0)^G|\}$ of vertices of $\Gamma_p(G)$. We claim that any two elements of this subset (in case of existence) are twins in $\Gamma_p(G)$. Since $x_0$ is a non-central power of $x_0y_0$, by Lemma \ref{power} it follows that $|(x_0y_0)^G|$ and $|x_0^G|$ are either equal or twins, and analogously $|(x_0y_0)^G|$ and $|y_0^G|$ are also either equal or twins. Moreover, being equal or twins is an equivalence relation on the vertex set of $\Gamma_p(G)$, so $|x_0^G|$ and $|y_0^G|$ are either equal or twins too. 

If $\mathcal{N}(|(x_0y_0)^G|)$ is a clique, then $\mathcal{N}(|(x_0y_0)^G|)\cup\{|(x_0y_0)^G|\}$ must be a connected component of $\Gamma_p(G)$ due to its regularity. However this graph is connected due to Lemma \ref{connected}, so we deduce that $\Gamma_p(G)$ is complete, which contradicts our assumptions. Consequently, there exist two vertices $|x_1^G|,|y_1^G|\in\mathcal{N}(|(x_0y_0)^G|)$, where $x_1,y_1\in G\smallsetminus\ze{G}$ are $p$-regular elements, such that $(|x_1^G|,|y_1^G|)=1$. Besides, in virtue of their primary decompositions and Lemma \ref{power}, it can be assumed that the orders of both $x_1$ and $y_1$ are prime powers.

Now we prove that, up to interchanging $x_1$ and $y_1$, it holds that $x_1$ is a $q$-element with $(|x_1^G|,qr)=r$. Observe that $q$ cannot divide both $|x_1^G|$ and $|y_1^G|$, so let us assume that $(|x_1^G|,q)=1$. It follows, up to conjugation, that $x_0\in\ce{G}{x_1}$. If $x_1$ is not a $q$-element, then there are powers of the $p$-regular element $x_0x_1\notin\ze{G}$ which yield $x_0$ and $x_1$, so by Lemma \ref{power} we get that $|(x_0x_1)^G|$, $|x_0^G|$ and $|x_1^G|$ are either equal or twins. But $|x_0^G|$ and $|(x_0y_0)^G|$ are also either equal or twins due to the first paragraph, so we deduce that $|x_1^G|$ and $|(x_0y_0)^G|$ are either equal or twins, which is not possible since $|y_1^G|\in \mathcal{N}(|(x_0y_0)^G|)$ and $(|x_1^G|,|y_1^G|)=1$. Thus $x_1$ is a $q$-element. Further, if $r$ does not divide $|x_1^G|$, then $x_1$ is a $r$-element by the same previous arguments, a contradiction. Hence $(|x_1^G|, qr)=r$, as wanted. The remaining claims concerning $y_1$ and its class size follow with similar reasonings.  

\noindent (b) We continue with the notation of statement (a). Let us suppose that $r$ does not divide $|x_0^G|$, so we have $x_0y_1=y_1x_0$ up to conjugation. In virtue of Lemma \ref{power} we obtain that $|x_0^G|$ and $|y_1^G|$ are either equal or twins. But $|x_0^G|$ and $|(x_0y_0^G)|$ are either equal or twins, as we saw at the beginning of the proof, so $|y_1^G|$ and $|(x_0y_0^G)|$ are either equal or twins, a contradiction. Hence $r$ divides $|x_0^G|$, and one can analogously prove that $q$ divides $|y_0^G|$. Now $qr$ divides $|(x_0y_0)^G|$ because this class size is divisible by $|x_0^G|$ and $|y_0^G|$.
\end{proof}

\bigskip

\begin{proof}[Proof of the Main Theorem]
Arguing by contradiction, we suppose that $\Gamma_p(G)$ is $k$-regular but non-complete. We divide the proof into the following three steps.

\smallskip

\noindent \textbf{\underline{Step 1.}} $|G/\ze{G}|_{p'}$ cannot be a power of a prime.

\smallskip

Let us suppose that $\pi(G/\ze{G})=\{p,q\}$, for some prime $q\neq p$. In virtue of Lemma \ref{prime_power}, we get that all non-central $p$-regular elements of $G$ do not have prime power class size, that is, $|g^G|$ is divisible by $p$ and $q$ for every vertex $|g^G|$ of $\Gamma_p(G)$. Hence the graph is complete, which is a contradiction.

\smallskip

\noindent \textbf{\underline{Step 2.}} If $q\neq r$ are prime divisors of $|G/\ze{G}|_{p'}$, then there are no elements $x_0,y_0\in G\smallsetminus \ze{G}$ such that $x_0$ is a $q$-element, $y_0$ is a $r$-element, and $x_0y_0=y_0x_0$.

\smallskip

Otherwise, in virtue of Proposition \ref{key} (a), there exists a $q$-element $x_1\in G\smallsetminus \ze{G}$ and a $r$-element $y_1\in G\smallsetminus\ze{G}$ such that $|x_1^G|,|y_1^G|\in\mathcal{N}(|(x_0y_0)^G|)$ and $(|x_1^G|,|y_1^G|)=1$ with $(|x_1^G|,qr)=r$ and $(|y_1^G|,qr)=q$. Set $A:=\mathcal{N}(|y_1^G|)\smallsetminus\{|(x_0y_0)^G|\}$. 

We claim that every element of $A$ is divisibly by either $q$ or $r$. Let us suppose that there exists $|g^G|\in A$ such that $(|g^G|, qr)=1$. Applying the primary decomposition of $g$ joint with Lemma \ref{power}, we may assume that $g$ is a $t$-element for some prime $t\neq p$. Furthermore, we may assume $t\neq q$ without loss of generality. Since $q$ does not divide $|g^G|$, then up to conjugation $gx_0=x_0g$, and in virtue of Lemma \ref{power} one can easily obtain that $|g^G|$ and $|x_0^G|$ are either equal or twins. If we reproduce the same arguments with $|x_1^G|$ instead of $|x_0^G|$, then we deduce that $|x_1^G|$ and $|g^G|$ are either equal or twins. Thus $|g^G|\in A=\mathcal{N}(|y_1^G|)$ leads to the contradiction $(|x_1^G|,|y_1^G|)\neq 1$.

Since we have proved that every element of $A$ is divisibly by either $q$ or $r$, then by Proposition \ref{key} (b) it holds that the degree of $|(x_0y_0)^G|$ is at least $|A|+|\{|x_1^G|,|y_1^G|\}|=|A|+2$. Nevertheless the degree of $|y_1^G|$ is clearly $|A|+1$, which is a contradiction due to the regularity of the graph.

\smallskip

\noindent \textbf{\underline{Step 3.}} The final contradiction.

\smallskip

Set $\pi(G/\ze{G})=\{p, q_1,q_2,...,q_n\}$, where $\{q_1,q_2,...,q_n\}$ are pairwise distinct primes, being all different from $p$. By Step 1 we may affirm $n\geq 2$. Let us show that for every $1\leq i\leq n$ and for every $q_i$-element $g_i\in G\smallsetminus \ze{G}$ it holds that $\{q_1,...,q_n\}\smallsetminus\{q_i\}\subseteq \pi(g_i^G)$ and $q_i\notin \pi(g_i^G)$. Arguing by contradiction, and w.l.o.g., we may first suppose that there is a $q_1$-element $g_1\in G\smallsetminus \ze{G}$ such that $|g_1^G|$ is not divisible by $q_2$. Then there exists a (non-central) Sylow $q_2$-subgroup $Q$ of $G$ such that $Q\leqslant \ce{G}{g_1}$, and so we can take a $q_2$-element $h\in Q\smallsetminus \ze{G}$. But this is not possible due to Step 2. Secondly, if $q_1$ divides $|g_1^G|$, then $\pi(g_1^G)$ contains $\pi(G/\ze{G})\smallsetminus\{p\}$, and the incompleteness of $\Gamma_p(G)$ implies the existence of some vertex $|a^G|$ with $(|a^G|,|g_1^G|)=1$, where $a\in G\smallsetminus\ze{G}$ is a $p$-regular element. This forces $|a^G|$ to be a $p$-power, and so $\Gamma_p(G)$ is complete by Lemma \ref{prime_power}, a contradiction. Observe that $\pi(g_i^G)$ may also contains $p$.

Let us consider that $n\geq 3$. Observe that if $g\in G\smallsetminus\ze{G}$ is a $p$-regular element, then its primary decomposition is $g=g_{k_1}\cdots g_{k_r}$ where $g_{k_j}$ is the $q_{k_j}$-part of $g$, with $\{k_1,...,k_r\}\subseteq \{1,...,n\}$. We may assume that $g_{k_j}$ is non-central for every $1\leq j\leq r$, since otherwise $|g^G|$ is invariant if we remove $g_{k_j}$ from the primary decomposition of $g$. Hence, $|g^G|$ and $|g_{k_j}^G|$ are either equal or twins by Lemma \ref{power}. Due to the assertion proved in the above paragraph, we have that the prime divisors of $|g_{k_j}^G|$ different from $p$ are $\{q_1,...,q_n\}\smallsetminus\{q_{k_j}\}$, for every $j\in \{1,...,r\}$. Thus, if $r>1$, then $\{q_1,...,q_n\}\subseteq \pi(g^G)$. Thus $|g^G|$ is a complete vertex of $\Gamma_p(G)$ because there is no non-trivial $p$-power class size in $G$ of a $p$-regular element via Lemma \ref{prime_power}, and it clearly follows that $\Gamma_p(G)$ is complete due to its regularity. On the other hand, if every $p$-regular element $g\in G\smallsetminus\ze{G}$ has prime power order, then it is easy to see that the assumption $n\geq 3$ also leads to the completeness of $\Gamma_p(G)$.

Consequently we may suppose $n=2$. In particular, we have that every $q_1$ element $g_1\in G\smallsetminus\ze{G}$ verifies $(|g_1^G|,q_1q_2)=q_2$, and every $q_2$-element $g_2\in G\smallsetminus\ze{G}$ verifies $(|g_2^G|,q_1q_2)=q_1$. Since there is no class size in $G$ of a $p$-regular element that is a non-trivial prime power in virtue of Lemma \ref{prime_power}, then both $|g_1^G|$ and $|g_2^G|$ are necessarily divisible by $p$. But this yields that both $|g_1^G|$ and $|g_2^G|$ are complete vertices of $\Gamma_p(G)$, and the regularity of this graph now leads to its completeness, the final contradiction.
\end{proof}


\bigskip

\noindent \textbf{Declarations.} The author declares no conflict of interest.



\begin{thebibliography}{99}


	

	\bibitem{BF1}
	\begin{sc} A. Beltrán and M.J. Felipe\end{sc}:
	On the diameter of a $p$-regular conjugacy class graph of finite groups,
	\emph{Comm. Algebra} \textbf{30} (2002) 5861--5873.
	
	
	\bibitem{BF2}
	\begin{sc} A. Beltrán and M.J. Felipe\end{sc}:
	Finite groups with a disconnected $p$-regular conjugacy class graph,
	\emph{Comm. Algebra} \textbf{32} (2004) 3503--3516.
	
	
	\bibitem{BFsurvey}
	\begin{sc} A. Beltrán and M.J. Felipe\end{sc}:
	\emph{Conjugacy classes of $p$-regular elements in $p$-solvable groups},
	Groups St. Andrews 2005, London Math. Soc. Lecture Note Series \textbf{339} (2007) 224--229, Cambridge University Press.
	
	
	\bibitem{BFM}
	\begin{sc} A. Beltrán, M.J. Felipe and C. Melchor\end{sc}:
	\emph{Graphs associated to conjugacy classes of normal subgroups in finite groups},
	\emph{J. Algebra} \textbf{443} (2015) 335--348.
	
	
	\bibitem{BCHP}
	\begin{sc} M. Bianchi, R.D. Camina, M. Herzog and E. Pacifici\end{sc}:
	Conjugacy classes of finite groups and graph regularity,
	\emph{Forum Math.} \textbf{27} (2015) 3167--3172.
	

	\bibitem{BHPS}
	\begin{sc} M. Bianchi, M. Herzog, E. Pacifici and G. Saffirio\end{sc}:
	On the regularity of a graph related to conjugacy classes of groups,
	\emph{European J. Combin.} \textbf{33} (2012) 1402--1407.
	
	
	\bibitem{many}
	\begin{sc} R.D. Camina, A. Maróti, E. Pacifici, C. Parker, K. Rekvényi, J. Saunders, V. Sotomayor, G. Tracey, M. van Beek\end{sc}:
	Groups with conjugacy classes of coprime sizes,
	\emph{accepted in Bull. London Math. Soc. } (2025). 
	
	
	\bibitem{FJS}
	\begin{sc} M.J. Felipe, M.K. Jean-Philippe and V. Sotomayor\end{sc}:
	Groups whose common divisor graph on $p$-regular classes has diameter three,
	\emph{Mediterr. J. Math.} \textbf{22} (2025), article 17.
	
	
	\bibitem{K}
	\begin{sc} L.S. Kazarin\end{sc}:
	On groups with isolated conjugacy classes,
	\emph{Izv. Vyssh. Uchebn. Zaved. Mat.} \textbf{7} (1981) 40--45; English translation in \emph{Soviet Math.} \textbf{25} (1981) 43--49.
	
	
	\bibitem{Lewis}
	\begin{sc} M.L. Lewis\end{sc}:
	An overview of graphs associated with character degrees and conjugacy class sizes in finite groups, 
	\emph{Rocky Mt. J. Math.} \textbf{38} (2008) 175--211.
	
	
	\bibitem{LZ}
	\begin{sc} Z. Lu and J. Zhang\end{sc}:
	On the diameter of a graph related to $p$-regular conjugacy classes of finite groups, 
	\emph{J. Algebra} \textbf{231} (2000) 705--712.	

	
	
	\bibitem{RA}
	\begin{sc} S.H. Rahimi and Z. Akhlaghi\end{sc}:
	On the regular graph related to the $G$-conjugacy classes,
	\emph{Bull. Aust. Math. Soc.} \textbf{105} (2022) 101--105.
	
	

	
\end{thebibliography}
\end{document}